\documentclass[11pt, a4paper]{amsart}
\usepackage[utf8]{inputenc}
\usepackage{amssymb,amsthm,amsmath, amsfonts}
\usepackage{color, graphicx, enumerate}

\title[On a Loomis-Whitney Type Inequality]{On a Loomis-Whitney Type Inequality for Permutationally Invariant Unconditional Convex Bodies}
\author{Piotr Nayar}
\address{Piotr Nayar \\ Institute of Mathematics, University of Warsaw, Banacha 2, 02-097 Warszawa, Poland.}
\email{nayar@mimuw.edu.pl}
\author{Tomasz Tkocz}
\address{Tomasz Tkocz \\ College of Inter-Faculty Individual Studies in Mathematics and Natural Sciences, and Institute of Mathematics, University of Warsaw, Banacha 2, 02-097 Warszawa, Poland.}
\email{t.tkocz@students.mimuw.edu.pl}


\newtheorem{thm}{Theorem}
\newtheorem{lm}{Lemma}

\theoremstyle{definition}
\newtheorem{remark}{Remark}

\newcommand{\R}{\mathbb{R}}

\newcommand{\E}{\mathbb{E}}
\newcommand{\dd}{\mathrm{d}}
\newcommand{\1}{\textbf{1}}
\newcommand{\fun}[3]{#1\colon #2 \longrightarrow #3}

\begin{document}

\begin{abstract}
For a permutationally invariant unconditional convex body $K$ in $\R^n$ we define a finite sequence $(K_j)_{j=1}^n$ of projections of the body $K$ to the space spanned by first $j$ vectors of the standard basis of $\R^n$. We prove that the sequence of volumes $(|K_j|)_{j=1}^n$ is log-concave.
\end{abstract}

\maketitle

\section{Introduction}\label{sec:intro}

The main interest in convex geometry is the examination of sections and projections of sets. Some introduction can be found in a monograph by Gardner, \cite{Gar-tomography}. We are interested in a class $\mathcal{PU}_n$ of convex bodies in $\R^n$ which are unconditional and permutationally invariant. 

Let us briefly recall some definitions. A convex body $K$ in $\R^n$ is called \emph{unconditional} if for every point $(x_1, \ldots, x_n) \in K$ and every choice of signs $\epsilon_1, \ldots, \epsilon_n \in \{-1, 1\}$ the point $(\epsilon_1 x_1, \ldots, \epsilon_n x_n)$ also belongs to $K$. A convex body $K$ in $\R^n$ is called \emph{permutationally invariant} if for every point $(x_1, \ldots, x_n) \in K$ and every permutation $\pi: \{1,  \ldots, n\} \longrightarrow \{1,  \ldots, n\}$ the point $(x_{\pi(1)}, \ldots, x_{\pi(n)})$ is also in $K$. A sequence $(a_i)_{i=1}^n$ of positive real numbers is called \emph{log-concave} if $a_i^2 \geq a_{i-1}a_{i+1}$, for $i = 2, \ldots, n-1$.

The main result of this paper reads as follows. 
\begin{thm}\label{thm1}
   Let $n \geq 3$ and let $K \in \mathcal{PU}_n$. For each $i = 1, \ldots, n$ we define a convex body $K_i \in \mathcal{PU}_i$ as an orthogonal projection of $K$ to the subspace $\{(x_1, \ldots, x_n) \in \R^n \ | \ x_{i+1} = \ldots = x_n = 0\}$. Then the sequence of volumes $(|K_i|)_{i=1}^n$ is log-concave. In particular
   \begin{equation}\label{eq0}
      |K_{n-1}|^2 \geq |K_n|\cdot |K_{n-2}|.
   \end{equation}
\end{thm}

Inequality \eqref{eq0} is related to the problem of negative correlation of coordinate functions on $K \in \mathcal{PU}_n$, i.e. the question whether for every $t_1, \ldots, t_n \geq 0$
\begin{equation}\label{eq1}
   \mu_K \left( \bigcap_{i=1}^n \{|x_i| \geq t_i\}\right) \leq \prod_{i=1}^n \mu_K \left(|x_i| \geq t_i \right),
\end{equation}
where $\mu_K$ is normalized Lebesgue measure on $K$. Indeed, the Taylor expansion of the function $h(t) = \mu_K(|x_1| \geq t)\mu_K(|x_2| \geq t) - \mu_K(|x_1| \geq t, |x_2| \geq t)$ at $t=0$ contains 
\[
  \frac{1}{|K_n|^2}\left( |K_{n-1}|^2 - |K_{n-2}|\cdot |K_n| \right)t^2,
\]
cf. \eqref{eq0}, as a leading term. The property \eqref{eq1}, the so-called concentration hypothesis and the central limit theorem for convex bodies are closely related, see \cite{An}. The last theorem has been recently proved by Klartag, \cite{Kla}. 

The negative correlation property in the case of generalized Orlicz balls was originally investigated by Wojtaszczyk in \cite{Onu}. A generalized Orlicz ball is a set
\[
B = \left\{ (x_1,\ldots,x_n) \in \R^n \; \big|  \; \sum_{i=1}^n f_i(|x_i|) \leq n   \right\},
\]
where $f_1,\ldots,f_n$ are some Young functions (see \cite{Onu} for the definition). In probabilistic terms Pilipczuk and Wojtaszczyk (see \cite{PW}) have shown that the random variable $X=(X_1,\ldots,X_n)$ uniformly distributed on $B$ satisfies the inequality
\[
\textrm{Cov}(f(|X_{i_1}|,\ldots,|X_{i_k}|), g(|X_{j_1}|,\ldots,|X_{j_l}|)) \leq 0
\]
for any bounded coordinate-wise increasing functions $\fun{f}{\R^k}{\R}$, $\fun{g}{\R^l}{\R}$ and any disjoint subsets $\{ i_1,\ldots,i_k \}$ and $\{ j_1,\ldots,j_l \}$ of $\{1,\ldots,n \}$. In the case of generalized isotropic Orlicz balls this result implies the inequality
\[
\textrm{Var}|X|^p \leq \frac{Cp^2}{n} \E |X|^{2p}, \qquad p \geq 2,
\]
from which some reverse H\"older inequalities can be deduced (see \cite{Fle}).

One may ask about an example of a nice class of Borel probability measures on $\R^n$ for which the negative correlation inequality hold. Considering the example of the measure with the density
\[
	p(x_1, \ldots, x_n) = \exp\left( -2(n!)^{1/n}\max\{|x_1|, \ldots, |x_n|\} \right),
\]
which was mentioned by Bobkov and Nazarov in a different context (see \cite[Lemma 3.1]{Bob}), we certainly see that the class of unconditional and permutationally invariant log-concave measures would not be the answer. Nevertheless, it remains still open whether the negative correlation of coordinate functions holds for measures uniformly distributed on the bodies from the class $\mathcal{PU}_n$.

We should remark that our inequality \eqref{eq0} is similar to some auxiliary result by Giannopoulos, Hartzoulaki and Paouris, see \cite[Lemma 4.1]{Pao}. They proved that a version of inequality \eqref{eq0} holds, up to the multiplicative constant $\frac{n}{2(n-1)}$, for an arbitrary convex body.

The paper is organised as follows. In Section \ref{sec:proof} we give the proof of Theorem \ref{thm1}. Section \ref{sec:remarks} is devoted to some remarks. Several examples are there provided as well.

\section{Proof of the main result}\label{sec:proof}

Here we deal with the proof of Theorem \ref{thm1}. We start with an elementary lemma.

\begin{lm}\label{lm1}
	Let $\fun{f}{[0, L]}{[0, \infty)}$ be a nonincreasing concave function such that $f(0) = 1$. Then
  \begin{equation} \label{eqlem}
   	\frac{n-1}{n}\left( \int_0^L f(x)^{n-2} \dd x \right)^2 \geq \int_0^L xf(x)^{n-2} \dd x, \qquad n \geq 3.
   \end{equation}
\end{lm}

\begin{proof}
By a linear change of a variable one can assume that $L=1$. Since $f$ is concave and nonincreasing, we have $1-x \leq f(x) \leq x$ for $x \in [0,1]$. Therefore, there exists a real number $\alpha \in [0,1]$ such that for $g(x) = 1 - \alpha x$ we have
\[
	\int_0^1 f(x)^{n-2} \dd x = \int_0^1 g(x)^{n-2} \dd x.
\] 
Clearly, we can find a number $c \in [0,1]$ such that $f(c) = g(c)$. Since $f$ is concave and $g$ is affine, we have $f(x) \geq g(x)$ for $x \in [0, c]$ and $f(x) \leq g(x)$ for $x \in [c, 1]$. Hence,
\begin{align*}
   \int_0^1 x(f(x)^{n-2} - g(x)^{n-2}) \dd x \leq &\int_0^c c(f(x)^{n-2} - g(x)^{n-2}) \dd x \\
   &\ + \int_c^1 c(f(x)^{n-2} - g(x)^{n-2}) \dd x \ = \ 0.
\end{align*}
We conclude that it suffices to prove (\ref{eqlem}) for the function $g$, which is by simple computation equivalent to
\begin{align*}
	\frac{1}{\alpha^2 n(n-1)}\left( 1 - (1 - \alpha)^{n-1} \right)^2 \geq \frac{1}{\alpha^2}\Bigg( &\frac{1}{n-1}\left( 1 - (1 - \alpha)^{n-1} \right) \\
&- \frac{1}{n}\left( 1 - (1 - \alpha)^n \right)\Bigg).
\end{align*}
To finish the proof one has to perform a short calculation and use Bernoulli's inequality.

\end{proof}

\begin{remark}
	A slightly more general form of this lemma appeared in \cite{GMR} and, as it is pointed out in that paper, the lemma is a particular case of a result of \cite[p. 182]{MOP}. Only after the paper was written we heard about these references from Prof. A. Zvavitch, for whom we are thankful. Our proof differs only in a few details, yet it is provided for the convenience of the reader.
\end{remark}

\begin{proof}[Proof of Theorem \ref{thm1}]
Due to an inductive argument it is enough to prove inequality \eqref{eq0}.
   
Let $\fun{g}{\R^{n-1}}{\{0, 1\}}$ be a characteristic function of the set $K_{n-1}$. Then, by permutational invariance and unconditionality, we have
\begin{equation}\label{eq2}
   |K_{n-1}| = 2^{n-1}(n-1)!\int_{x_1 \geq \ldots \geq x_{n-1} \geq 0} g(x_1, \ldots, x_{n-1}) \dd x_1 \ldots \dd x_{n-1},
\end{equation}
and similarly
\begin{equation}\label{eq3}
	|K_{n-2}| = 2^{n-2}(n-2)!\int_{x_1 \geq \ldots \geq x_{n-2} \geq 0} g(x_1, \ldots, x_{n-2}, 0) \dd x_1 \ldots \dd x_{n-2}.
\end{equation}
Moreover, permutational invariance and the definition of a projection imply
\begin{equation}\label{eq5}
	\1_{K_n}(x_1, \ldots, x_n) \leq \prod_{i=1}^n g(x_1, \ldots, \hat{x_i}, \ldots, x_n).
\end{equation}
Thus
\begin{equation}\label{eq7}
\begin{split}
	|K_n| &\leq 2^nn!\int_{x_1 \geq \ldots \geq x_n \geq 0} \prod_{i=1}^n g(x_1, \ldots, \hat{x_i}, \ldots, x_n)\dd x_1 \ldots \dd x_n  \\
	&= 2^nn!\int_{x_1 \geq \ldots \geq x_n \geq 0} g(x_1, \ldots, x_{n-1}) \dd x_1 \ldots \dd x_n \\
	&= 2^nn!\int_{x_1 \geq \ldots \geq x_{n-1} \geq 0} x_{n-1}g(x_1, \ldots, x_{n-1}) \dd x_1 \ldots \dd x_{n-1},
\end{split}
\end{equation}
where the first equality follows from the monotonicity of the function $g$ for nonnegative arguments with respect to each coordinate. We define a function $\fun{F}{[0, \infty)}{[0, \infty)}$ by the equation
\[
	F(x) = \frac{\int_{x_1 \geq \ldots \geq x_{n-2} \geq x} g(x_1, \ldots, x_{n-2}, x) \dd x_1 \ldots \dd x_{n-2}}{\int_{x_1 \geq \ldots \geq x_{n-2} \geq 0} g(x_1, \ldots, x_{n-2}, 0) \dd x_1 \ldots \dd x_{n-2}}.
\]
One can notice that
\begin{enumerate}[1.]
   \item $F(0) = 1$.
   \item The function $F$ is nonincreasing as so is the function
   \[
   	x \mapsto g(x_1, \ldots, x_{n-2}, x)\1_{\{x_1 \geq \ldots \geq x_{n-2} \geq x\}}.
   \]
   \item The function $F^{1/(n-2)}$ is concave on its support $[0, L]$ since $F(x)$ multiplied by some constant equals the volume of the intersection of the convex set $K_{n-1} \cap \{x_1 \geq \ldots \geq x_{n-1} \geq 0\}$ with the hyperplane $\{x_{n-1} = x\}$. This is a simple consequence of the Brunn-Minkowski inequality, see for instance \cite[page 361]{Gar}.
\end{enumerate}
By the definition of the function $F$ and equations \eqref{eq2}, \eqref{eq3} we obtain
\[
	\int_0^L F(x) \dd x = \frac{\frac{1}{2^{n-1}(n-1)!}|K_{n-1}|}{\frac{1}{2^{n-2}(n-2)!}|K_{n-2}|} = \frac{1}{2(n-1)}\cdot\frac{|K_{n-1}|}{|K_{n-2}|},
\]
and using inequality \eqref{eq7}
\[
	\int_0^L xF(x) \dd x \geq  \frac{\frac{1}{2^nn!}|K_n|}{\frac{1}{2^{n-2}(n-2)!}|K_{n-2}|} = \frac{1}{2^2n(n-1)}\cdot\frac{|K_n|}{|K_{n-2}|}.
\]

Therefore it is enough to show that
\[
	\frac{n-1}{n}\left( \int_0^L F(x) \dd x \right)^2 \geq \int_0^L xF(x) \dd x.
\]
This inequality follows from Lemma \ref{lm1}.
\end{proof}

\section{Some remarks}\label{sec:remarks}

In this section we give some remarks concerning Theorem \ref{thm1}.

\begin{remark}\label{rem1}

Apart from the trivial example of the $B_\infty^n$ ball, there are many other examples of bodies for which equality in \eqref{eq0} is attained. Indeed, analysing the proof, we observe that for the equality in \eqref{eq0} the equality in Lemma \ref{lm1} is needed. Therefore, the function $F^{1/(n-2)}$ has to be linear and equal to $1 - x$. Taking into account the equality conditions in the Brunn-Minkowski inequality (consult \cite[page 363]{Gar}), this is the case if and only if the set $K_{n-1} \cap \{x_1 \geq \ldots \geq x_{n-1} \geq 0\}$ is a cone $C$ with the base $(K_{n-2} \cap \{x_1 \geq \ldots \geq x_{n-2} \geq 0\})\times \{0\} \subset \R^{n-1}$ and the vertex $(z_0, \ldots, z_0) \in \R^{n-1}$. Thus if for a convex body $K \in \mathcal{PU}_n$ we have the equality in \eqref{eq0}, then this body $K$ is constructed in the following manner. Take an arbitrary $K_{n-2} \in \mathcal{PU}_{n-2}$. Define the set $K_{n-1}$ as the smallest permutationally invariant unconditional body containing $C$. For $z_0$ from some interval the set $K_{n-1}$ is convex. For the characteristic function of the body $K$ we then set $\prod_{i=1}^n \1_{K_{n-1}}(x_1, \ldots, \hat{x_i}, \ldots, x_n)$.
   
A one more natural question to ask is when a sequence $(|K_i|)_{i=1}^n$ is geometric
Bearing in mind what has been said above for $i=2, 3, \ldots, n-1$ we find that a sequence $(|K_i|)_{i=1}^n$ is geometric if and only if
\[
   	K = [-L, L]^n \ \cup \bigcup_{i \in \{1, \ldots, n\}, \epsilon \in \{-1, 1\}} \textrm{conv} \left\{ \epsilon a e_i, \{x_i = \epsilon L, |x_k| \leq L, k \neq i\} \right\},
   \]
for some positive parameters $a$ and $L$ satisfying $L < a < 2L$, where $e_1, \ldots, e_n$ stand for the standard orthonormal basis in $\R^n$. One can easily check that $|K_i| = 2^iL^{i-1}a$.
\end{remark}

\begin{remark}\label{rem2}
Suppose we have a sequence of convex bodies $K_n \in \mathcal{PU}_n$, for $n \geq 1$, such that $K_n = \pi_{n}(K_{n+1})$, where by $\fun{\pi_n}{\R^{n+1}}{\R^n}$ we denote the projection $\pi_n(x_1, \ldots, x_n, x_{n+1}) = (x_1, \ldots, x_n)$. Since Theorem \ref{thm1} implies that the sequence $(|K_n|)_{n=1}^\infty$ is log-concave we deduce the existence of the limits
\[
  \lim_{n\to\infty} \frac{|K_{n+1}|}{|K_n|}, \qquad \lim_{n\to\infty} \sqrt[n]{|K_n|}.
\]
We can obtain this kind of sequences as finite dimensional projections of an Orlicz ball in $\ell_\infty$.
\end{remark}

\section*{Acknowledgements}

The authors would like to thank Prof. K.~Oleszkiewicz for a valuable comment regarding the equality conditions in Theorem \ref{thm1} as well as Prof. R.~Lata{\l}a for a stimulating discussion.


\begin{thebibliography}{9}


\bibitem{An} M. Anttila, K. Ball and I. Perissinaki, \emph{The central limit problem for convex bodies}, Trans. Amer. Math. Soc., \textbf{355} (2003), pp. 4723-4735.

\bibitem{Bob} S. G. Bobkov, F. L. Nazarov, \emph{On convex bodies and log-concave probability measures with unconditional basis}, Geometric aspects of functional analysis, 53--69, Lecture Notes in Math., 1807, Springer, Berlin, 2003.

\bibitem{Fle} B. Fleury, \emph{Between Paouris concentration inequality and variance conjecture}, Ann. Inst. Henri Poincaré Probab. Stat. \textbf{46} (2010), no. 2, 299--312.

\bibitem{Gar-tomography} R. J. Gardner, \emph{Geometric Tomography}, Encyclopedia of Mathematics, Vol. 58, Cambridge University Press, 2006.

\bibitem{Gar} R. J. Gardner, \emph{The Brunn-Minkowski inequality}, Bull. Amer. Math. Soc., \textbf{39} (2002), 355--405.

\bibitem{GMR} Y. Gordon, M. Meyer, S. Reisner, \emph{Zonoids with minimal volume-product—a new proof}, Proc. Amer. Math. Soc. \textbf{104} (1988), no. 1, 273--276.

\bibitem{Pao} A. Giannopoulos, M. Hartzoulaki, G. Paouris, \emph{On a local version of the Aleksandrov-Fenchel inequalities for the quermassintegrals of a convex body}, Proc. Amer. Math. Soc. \textbf{130} (2002), 2403--2412.

\bibitem{Kla} B. Klartag, \emph{A central limit theorem for convex sets}, Invent. Math., Vol. \textbf{168}, (2007), 91--131.

\bibitem{MOP} W. Marshall, I. Olkin, F. Proschan, \emph{Monotonicity of ratios of means and other applications of majorization}, Inequalities (O. Shisha, ed.), Academic Press, New York, 1967.

\bibitem{PW} M. Pilipczuk, J. O. Wojtaszczyk, \emph{The negative association property for the absolute values of random variables equidistributed on a generalized Orlicz ball}, Positivity \textbf{12} (2008), no. 3, 421--474.

\bibitem{Onu} J. O. Wojtaszczyk, \emph{The square negative correlation property for generalized Orlicz balls}, Geometric Aspects of Functional Analysis, 305--313, Lecture Notes in Math., 1910, Springer, Berlin, 2007.


\end{thebibliography}
\end{document}